\documentclass{birkmult}
 \newtheorem{thm}{Theorem}[section]
 \newtheorem{cor}[thm]{Corollary}
 \newtheorem{lem}[thm]{Lemma}
 \newtheorem{prop}[thm]{Proposition}
\newtheorem*{pr}{Proof of the theorem}
\theoremstyle{definition}
 \newtheorem{defn}[thm]{Definition}
 \theoremstyle{remark}
 \newtheorem{rem}[thm]{Remark}
 
 \numberwithin{equation}{section}

\begin{document}

\title[Solutions of Nonlinear DE]
 {Solutions of Nonlinear Differential Equations}
\author[N.~Bedziuk]{Nadzeya Bedziuk}

\address{%
Department of Functional Analysis\\
Belarusian State University\\
Nezavisimosti 4\\
220050 Minsk\\
Belarus}

\email{nbedyuk@gmail.com}


\author[A.~Yablonski]{Aleh Yablonski}
\address{Department of Functional Analysis\\
Belarusian State University\\
Nezavisimosti 4\\
220050 Minsk\\
Belarus}

\email{yablonski@bsu.by}


\date{July 15, 2008}

\begin{abstract}
We consider an ordinary nonlinear differential equation with
generalized coefficients as an equation in differentials in
algebra of new generalized functions. Then the solution of such
equation will be a new generalized function. In the article we
formulate necessary and sufficient conditions when the solution of
the given equation in algebra of new generalized functions is
associated with an ordinary function. Moreover the class of all
possible associated functions was described.\end{abstract}

\maketitle
\section{Introduction}
The dynamics of many real systems or phenomena can be described by
nonlinear differential equations with generalized coefficients.
Unfortunately the theory of generalized functions allows only to formulate such equations and it is inapplicable to
solution of them. Therefore different interpretations of the
solution of nonlinear differential equations were proposed by many
mathematicians. In general, different interpretations of the same
equation lead to different solutions. See, e.g.,
\cite{A,DS,F,Li,P,Z}.  In order to choose an adequate interpretation
of the given equation one has to consider the reasons that are used
for modelling the dynamics of the real system.

In this paper we will consider the following nonlinear equation with generalized coefficients
\begin{equation}
\label{eq: 1}
 \dot{x}(t)=f(t,x(t))\dot{L}(t),
\end{equation}
where $ t \in [a;b] \subset \mathbb{R}$, $\dot{L}(t)$ is a derivative in the distributional sense. Generally, since $\dot{L}(t)$ is a distribution
and function $f(t,x(t))$ is not smooth then the product $f(t,x(t))\dot{L}(t)$ is not well defined and the solution of the equation (\ref{eq: 1})
essentially depends on the interpretation.

We will investigate equation (\ref{eq: 1}) by using the algebra of mnemofunctions (new generalized functions). It is worth mentioning that the first algebra of new generalized functions was proposed by J.F.~Colombeau in \cite{C}. Definitions of other algebras one can find in
\cite{E, R}. The general methods of construction of such algebras were proposed by A.B.~Antonevich and Ya.V.~Radyno in \cite{AR}.

In this paper we interpret equation (\ref{eq: 1}) as an equation
in differentials in algebra of new generalized functions from
\cite{L}. Such approach allows us to investigate ordinary and
stochastic differential equations in the unique way \cite{LSY,LY}.
Algebraic interpretation states that solution of the equation
(\ref{eq: 1}) is an element of algebra of new generalized
functions.
Naturally a problem arises to find conditions on coefficients of
equation in differentials which allow us to associate an ordinary
function with the solution of this equation. If such an ordinary function
exists then we call it solution of the equation (\ref{eq: 1}). The
sufficient conditions when the solution exists and the set of
possible solutions in this sense were described in article
\cite{YN}. In this article we will prove necessary conditions
for some class of coefficients.

\section{The algebra of mnemofunctions}
In this section we recall the definition of the algebra of mnemofunctions from \cite{L}, see also \cite{LSY} and
\cite{YN}.

At first we define an extended real line $\widetilde{\mathbb{R}} $ using a construction typical for non-standard analysis. Let $\overline{\mathbb{R}}
= \{(x_n)_{n=1}^\infty : x_n \in \mathbb{R} $ for all $n \in \mathbb{N} \}$ be a set of real sequences. We will call two sequences $\{x_n\} \in
\overline{\mathbb{R}}$ and $\{y_n\} \in \overline{\mathbb{R}}$  equivalent if there is a natural number $N$ such that $x_n =y_n$ for all $n > N$. The
set $\widetilde{\mathbb{R}}$ of equivalence classes will be called the extended real line and any of the classes a generalized real number.

It is easy to see that $\mathbb{R} \subset \widetilde{\mathbb{R}}$ as one may associate with any ordinary number $x \in \mathbb{R}$ a class
containing a stationary sequence $x_n = x$. The product $\widetilde x \widetilde y$
of two generalized real numbers is defined as the class of sequences equivalent to the sequence $\{x_n y_n\}$,  where $\{x_n\}$ and $\{y_n\}$ are the
arbitrary representatives of the classes $\widetilde x$ and $ \widetilde y$ respectively. It is evident that $\widetilde{\mathbb{R}}$ is an algebra. For any segment $\mathbf{T} =[a;b] \subset \mathbb{R}$ one can construct an extended segment $\widetilde{\mathbf{T}}$ in a similar way.

Consider the set of sequences of infinitely differentiable functions $\{f_n(x)\}$ on $\mathbb{R}$. We will call two sequences $\{f_n (x)\}$ and
$\{g_n (x)\}$ equivalent if there is a natural number $N$ such that $f_n (x) = g_n (x)$ for all $n > N$ and $x \in \mathbb{R}$. The set of classes of
equivalent functions is denoted by $ \mathcal{G} (\mathbb{R})$ and its elements are called mnemofunctions. Similarly one can define the space
$\mathcal{G} (\mathbf{T})$ for any interval $\mathbf{T}=[a;b]$. If we endow all these spaces with natural operations of addition and multiplication then they become algebras.

For each distribution $f$ we can construct a sequence $f_n$ of smooth functions such that $f_n$ converges to $f$ (e.g. one can consider the
convolution of $f$ with some $\delta$-sequence). This sequence defines the mnemofunction which corresponds to the distribution $f$. Thus the space of
distributions is a subset of the algebra of mnemofunctions. However, in this case, infinite set of mnemofunctions correspond to one distribution
(e.g., by taking different $\delta$-sequences). We will say that mnemofunction $\widetilde f = [\{f_n\}]$ is associated to a function $f$ from some topological space if $f_n$ converges to $f$ in this space.

Let $\widetilde f =[\{f_n(x)\}] $ and $\widetilde g =[\{g_n(x)\}] $ be mnemofunctions. Then there is a composition defined
$\widetilde f \circ \widetilde g = [\{f_n(g_n(x))\}]\in \mathcal{G} (\mathbb{R})$. In the same way one can define the value of mnemofunction
$\widetilde f $ at the generalized real point $\widetilde x =[\{x_n\}] \in \widetilde{\mathbb{R}}$ as $\widetilde f(\widetilde x) = [\{f_n(x_n)\}]$.

Let $H$
denote the subset of $\widetilde{\mathbb{R}} $ of nonnegative ``infinitely small numbers'':
$$H=\{ \widetilde{h} \in \widetilde{\mathbb{R}} :
\widetilde{h} = [\{h_n\}],\ h_n >0,  \lim h_n =0 \}.$$
For each $\widetilde h=[\{h_n\}] \in H $ and $\widetilde f = [\{f_n(x)\}] \in \mathcal{G} (\mathbb{R})$ we define a differential
${\mathrm d}_{\widetilde h} \widetilde f \in \mathcal{G}(\mathbb{R})$ by ${\mathrm d}_{\widetilde h} \widetilde f = [\{f_n(t + h_n) - f_n(t)\}]$. The construction
of the differential was proposed by N.V.~Lazakovich (see \cite{L}).

\section{Main results}
In this section we will formulate main results of this article.

Using introduced algebras now we can give an interpretation of equation (\ref{eq: 1}).
We replace ordinary functions in equation (\ref{eq: 1}) by corresponding mnemofunctions  and then
write the algebra's differentials. So we have
\begin{equation}
  \label{eq: diff}
 {\mathrm d}_{\widetilde h } \widetilde X(\widetilde t) = \widetilde f(\widetilde t, \widetilde X(\widetilde t)) {\mathrm d}_{\widetilde h }\widetilde L(\widetilde t),
\end{equation}
with initial value $\widetilde X |_{[\widetilde a; \widetilde h)}=\widetilde X^0$, where $\widetilde h=[\{h_n\}] \in H$, $\widetilde a = [\{a\}] \in
\widetilde{\mathbf{T}}$, $\widetilde t=[\{t_n\}] \in \widetilde{\mathbf{T}}$, $ \widetilde X = [\{X_n\}],\ \widetilde f=[\{f_n\}],\ \widetilde X^0
=[\{X^0_n\}],\ \widetilde L= [\{L_n\}]$ are elements of $\mathcal{G}(\mathbb{R})$.
Moreover $\widetilde f$ and $\widetilde L$ are associated to $f$ and $L$ respectively. If $\widetilde X$ is associated to some function $X$
then we say that $X$ is a solution of the equation (\ref{eq: 1}).

As it was shown in \cite{LSY} there exists a unique solution of equation (\ref{eq: diff}).
The purpose of the present paper is to investigate when the solution $\widetilde X$ of the equation (\ref{eq: diff}) converges to some ordinary function and to describe all possible limits.

Let $L(t)$, $t \in \mathbf{T} =[a;b]$ be a right-continuous function of a finite variation. We will assume that $L(t)=L(b)$ if $t > b$ and $L(t)=L(a)$
if $t<a$. Denote by $V_u^v L$ the total variation of function $L$ on the interval $[u;v] \subset \mathbf{T}$. Suppose that $f$ is Lipschitz
continuous function of a bounded growth, i.e. there exists constant $K$ such that for all $x \in \mathbb{R}$ and $t \in \mathbf{T}$
\begin{equation}
\label{eq: fg}
|f(t,x)|\le K (1+|x|).
\end{equation}

Consider the following convolutions with
$\delta$-sequence as representatives of mnemofunction  $\widetilde L$ from equation (\ref{eq: diff}):
\begin{equation} \label{eq: Ln}
L_n (t) = (L*\rho_n)(t)=\int_0^{1/n} L(t+s) \rho _n (s) {\mathrm d}s,
\end{equation}
where $\rho_n \in C^\infty (\mathbb{R})$, $\rho_n \ge 0$, $\mathrm{supp} \ \rho_n \subseteq [0;1/n]$ and $\int_0^{1/n} \rho_n (s){\mathrm d}s =1$.

In the same way we set
\begin{equation} \label{eq: fn}
f_{n}(t,x)=(f*\widetilde{\rho}_{n})(t,x)=\int_{[0,1/n]^{2}}f(t+u, x+v)\widetilde{\rho}_{n}(u,v){\mathrm d}u  {\mathrm d}v,
\end{equation}
where $\widetilde{\rho}_{n}\in C^{\infty}(\mathbb{R}^{2})$, $\widetilde{\rho}_{n} \geq 0$, $\mathrm{supp}\ \widetilde{\rho}_{n} \subseteq [0,
1/n]^{2}$, $ \int_{[0,1/n]^{2}}\widetilde{\rho}_{n}(u,v){\mathrm d}u {\mathrm d}v=1$.

By using representatives we can rewrite equation (\ref{eq: diff}) in the following form:
\begin{equation}
  \label{eq:representetieves system}
\left\{
\begin{array}{l}
 x_n(t+h_n)-x_n(t)=f_n(t,x_n(t))(L_n(t+h_n)-L_n(t)),\\
 x_n|_{[a;a+h_n)}(t)= x_n^0 (t).
\end{array}
\right.
\end{equation}

The solution $\widetilde x$ of the equation (\ref{eq: diff}) is associated with some function if and only if the sequence of the solutions $x_n$ of the
equation (\ref{eq:representetieves system}) converges. Therefore we have to investigate the limiting behavior of the sequence $x_n$.

Let $t$ be an arbitrary point of $\mathbf{T}$. There exists $m_t \in \mathbb{N}$ and $\tau_t \in [a;a+h_n)$ such that $t = \tau_t + m_t h_n$. Set
$t_k = \tau_t + k h_n$, $k = 0,1,\dots, m_t$. Then the solution of the equation (\ref{eq:representetieves system}) can be written as
\begin{equation}
\label{eq:representetieves explicit}
x_n(t)=x_n^0(\tau_t )+\sum_{k=0}^{m_t-1} f_n(t_k,x_n(t_k))(L_n(t_{k+1})-L_n(t_k)).
\end{equation}

Consider the function $F_n : [-\infty; +\infty] \to [0;1]$ given by
\begin{equation}
\label{eq: Fn}
F_n(x)=\int_x^{1/n} \rho_n(s){\mathrm d} s.
\end{equation}
Since $\rho_n (s) \ge 0$, then $F_n$ is a non-increasing function, $0 \le F_n (x) \le 1$ and $F_n(+ \infty) = 0$, $F_n(-\infty) = 1$. Denote by
$F_n^{-1}$ the inverse function of $F_n$, i.e., $F_n^{-1}: [0;1] \to [-\infty; +\infty]$ and
\begin{equation}
\label{eq: Fn-1}
F_n^{-1}(u)=\sup\left\{x:F_n(x)=u\right\}.
\end{equation}

In order to describe the limits of the sequence $x_n$ we consider the integral equation
\begin{equation}
\label{eq:limit int}
x(t) = x^0 + \int_a^t f(s,x(s)) {\mathrm d}L^c(s) + \sum_{a < s \le t} \Bigl(\varphi(\Delta L(s)f, x(s-), 1) - x(s-) \Bigr),
\end{equation}
where $t \in\mathbf{T}$, $L^c$ is a continuous part of the function $L$, $\Delta L(s) = L(s+) - L(s-)$ is a size of the jump of the function $L$ at
the epoch $s$, and for any function $z$ and $x\in\mathbb{R}$, $\varphi(z, x, u)$ denotes the solution of the integral equation
\begin{equation}
\label{eq: phi}
\varphi(z, x, u) = x + \int_{[0;u)} z (\varphi(z, x, v))\mu({\mathrm d}v),
\end{equation}
and $\mu(du)$ is a  probability measure defined on the Borel subsets of the interval $[0;1]$.

As it was shown in \cite{YN} there exists a unique solution of equation (\ref{eq:limit int}) if $f$ is Lipschitz continuous function.

\begin{defn}
 We say that the function $\sigma: [0;1] \to [0;1]$ belongs to class $G$ if there is a system of pairwise-disjoint intervals $(a_i;b_i] \subseteq [0;1]$, $i \in I$ such that
\begin{equation}
\label{eq: class G}
\sigma(u)= \left \{\begin{array}{l}
b_i, \ u \in (a_i;b_i], \\ u, \ u \notin \bigcup_{i \in I}
(a_i;b_i]. \\
\end{array} \right.
\end{equation}
\end{defn}

The following theorems describe the limits of the sequence $X_n$.
\begin{thm}[\cite{YB}]\label{theorem:suficient}
 Let $f$ be Lipschitz function satisfying (\ref{eq: fg}), $L$ is a
right-continuous function of a finite variation. Suppose $\int_{t \in \mathbf{T}} |x_n^0(\tau_t)-x^0| {\mathrm d} t \to 0$ and $F_n(F_n^{-1}(u)-\delta h_n ) \to
\sigma (u)$ as $n \to \infty$ and $h_n \to 0$ for all $\delta \in (0;1)$ and for any continuity point $u \in [0;1]$ of $\sigma$. Then $\sigma$
belongs to class $G$ and $$\int_{\mathbf{T}} |x_n(t)-x(t)| {\mathrm d} t \to 0$$ as $n \to \infty$ and $h_n \to 0$, where $x_n(t)$ is a solution of
equation (\ref{eq:representetieves system}) and $x(t)$ is a solution of the equation (\ref{eq:limit int}) with the measure $\mu$ generated by function $\sigma$.
\end{thm}

\begin{thm}\label{theorem:necessary}
Let $L$ be right-continuous function of a finite variation. Suppose for each Lipschitz function $f$ satisfying (\ref{eq: fg}),
solution of equation (\ref{eq:representetieves system}) $X_n$ converges in $\mathbf{L_1}(\mathbf{T})$ as $n\rightarrow \infty$, $h_{n}\rightarrow 0$.

If function $L$ continuous, then the limit of $x_n$ is a solution of equation (\ref{eq:limit int}). If $L$ is discontinuous, then there exists function
$\sigma \in G$ such that $F_{n}(F^{-1}_{n}(u)-\delta h_{n})\rightarrow \sigma(u)$ as $n\rightarrow \infty$, $h_{n}\rightarrow 0$, for all
$\delta\in (0,1)$ and $u \in [0,1]$ - continuity points of $\sigma$, and the limit of $x_n$ is a solution of equation (\ref{eq:limit int}) with measure
$\mu$ generated by function $\sigma$.
\end{thm}

\section{Auxiliary statements}

Suppose that for any $t \in \mathbf{T}$ and $n \in \mathbb{N}$ partition of interval $[0;1]$ is given: $0=\xi^n_0(t) \leq \xi^n_1(t)  \leq \ldots \leq \xi^n_{p+2}(t)=1$, where $p$ depends on $n$. Consider the recurrent sequence $\varphi^n_{k}(t) = \varphi^n_k(z,x,t)$ for $t \in \mathbf{T}, k=0,1,\ldots,p+2, n \in \mathbb{N}, x \in \mathbb{R}$, $z$ - Lipschitz function.
\begin{equation}\label{eq: varphi}
\left\{
\begin{array}{ll}
    \varphi^n_{k+1}(t)=\varphi^n_k(t)+z(\varphi^n_k(t))(\xi^n_{k+1}(t) - \xi^n_{k}(t))\\
    \varphi^n_0(t)=x.\\
\end{array}
\right.
\end{equation}
For each $u \in [0;1]$, $t \in \mathbf{T}$, $n=1,2,\ldots$ define $\sigma^n(u,t)$ and $\phi^n(u,t)$ as follows
\begin{align}\label{eq: sigma^n}
\sigma^n(u,t)=\left\{
\begin{array}{ll}
    \xi^n_{k}(t), \xi^n_{k-1}(t)<u \leq \xi^n_{k}(t)\\
    0, u=0\\
\end{array}
\right.
\end{align}
\begin{align}\label{eq: phi^n}
\phi^n(u,t)=\left\{
\begin{array}{ll}
    \varphi^n_k(t), \xi^n_{k-1}(t)<u \leq \xi^n_{k}(t)\\
    x, u=0.\\
\end{array}
\right.
\end{align}
It is easy to see that $\phi_n$ is a solution of the following integral equation.
$$
\phi^n(u,t)=x+\int_{[0;u)}z(\phi^n(s,t))\sigma^n({\mathrm d}s,t).
$$

\begin{lem}[{\cite{YN}}]\label{mainLemma}
Suppose that there is a nonincreasing left-continuous function $\sigma(u), u \in [0;1]$ such that
$$
\int_{\mathbf{T}}\left | \sigma^n(u,t)-\sigma(u) \right | {\mathrm d}t \to 0
$$
as $n \to \infty$ for any continuity point $u$ of $\sigma$.

Then $\sigma$ belongs to class $G$ and
$$
\int_{\mathbf{T}}\left | \phi^n(u,t)-\phi(u) \right | {\mathrm d}t \to 0
$$
as $n \to \infty$ for any continuity point $u$ of $\phi$, where $\phi(u)$ is a solution of the equation
\begin{equation}\nonumber
\phi(u)=x+\int_{[0;u)}z(\phi(s)){\mathrm d} \sigma(s).
\end{equation}
\end{lem}

For any fixed point $t$, let us denote $\phi^n(u,t)$ and $\sigma^n(u,t)$ as $\phi^n(u)$ and $\sigma^n(u)$ respectively.
\begin{lem}
\label{lemma: obratnaya} Suppose that there exists $x \in \mathbb{R}$ such that for each Lipschitz function $z:\mathbb{R} \to \mathbb{R}$ of a bounded growth number sequence $\phi^n(1) = \varphi_{p_n}^n(z,x)$ converges as $n \to \infty$. Then sequence of measures generated by functions $\sigma^n$ converges weakly.
\end{lem}

\begin{proof}
By definition of $\sigma^n(u)$ it is evident that $\sigma^n(u) \geq u$.
It means that we can represent interval $[0;1]$ as disjoint union of sets
$A$, $B$ and $C$, where
$$A=\{u:\lim_{n \to \infty}\sigma^n(u)=u\},$$
$$B=\{u:\overline{\lim}_{n \to \infty}\sigma^n(u) \geq \underline{\lim}_{n \to \infty}\sigma^n(u)>u\},$$
$$C=\{u:\overline{\lim}_{n \to \infty}\sigma^n(u) > \underline{\lim}_{n \to \infty}\sigma^n(u)=u\}.$$
Let us consider Lipschitz functions $z_{q,\varepsilon}$ for $q,\varepsilon \in \mathbb{Q}$, $\varepsilon >0$
defined as follows:
\begin{equation}\nonumber
z_{q,\varepsilon}(p)=\left\{
\begin{array}{ll}
    1,& p \leq x + q\\
    \frac{1}{\varepsilon}(x + q +\varepsilon -p),& x + q <p\leq x + q + \varepsilon\\
    0,& p > x + q + \varepsilon.\\
\end{array}
\right.
\end{equation}

Consider an arbitrary point $u \in B$. By definition of set $B$ we have that $\underline{\lim}_{n \to
\infty}\sigma^n(u)=b>u$. For each $n$ there exists such a number $m$ that
$\xi^n_{m-1}<u \leq \xi^n_{m}$. Then for large enough $n$ we have $\sigma^n(u)=\xi^n_{m} \geq
(b+u)/2$. In this case for each function $z_{q,\varepsilon}$
 such that $u \leq q < q + \varepsilon \leq (b+u)/2$,
$\phi^n(1) = x + \xi^n_{m}.$ Since $\phi^n(1)$ converges then $\sigma^n(u)=\xi^n_{m}$ converges also. At the same time convergence of $\sigma^n(u)$ leads to the following equality:
\begin{equation}\label{caseB}
\overline{\lim}_{n \to \infty}\sigma^n(u) = \underline{\lim}_{n
\to \infty}\sigma^n(u)=b>u.
\end{equation}

Consider an arbitrary point $u \in C$. By the definition of the set $C$ we have $u=\underline{\lim}_{n \to
\infty}\sigma^n(u) < \overline{\lim}_{n \to \infty}\sigma^n(u)
=d$. As $\overline{\lim}_{n \to \infty}\sigma^n(u) =d$, then
there is a subsequence $\{n_k\}$, such that $\lim_{n_k \to
\infty}\sigma^{n_k}(u)=d$. Suppose that there exists subsequence $\{n_i\}$, such that $\lim_{n_i \to
\infty}\sigma^{n_i}(u)=s$, where $u<s<d$. But if we take subsequence $\{n_k\}\cup\{n_i\}$ then the same arguments as in (\ref{caseB}) imply that $\lim_{n_i \to \infty}\sigma^{n_i}(u)=b$, what leads to the contradiction with the definition of $n_i$. Thus we can represent sequence $\{n\}$ as disjoint union $\{n'\}\sqcup\{n''\}$, where $\lim_{n'' \to
\infty}\sigma^{n''}(u)=d$ and $\lim_{n' \to
\infty}\sigma^{n'}(u)=u$.

According to Helly's selection principle $\sigma^n$ is a weakly-compact sequence, it means that from each subsequence of $\sigma^n$ one can select subsequence that converges weakly. Let us show that the whole sequence $\sigma^n$ also converges weakly.

Suppose there exist two subsequences that converge to different functions: $\sigma^{n*} \to \sigma^*$, $\sigma^{n**} \to \sigma^{**}$ and
 $\sigma^{*} \neq \sigma^{**}$. Using lemma \ref{mainLemma}
we obtain that $\sigma^{*} $ and $ \sigma^{**}$ belong to class $G$,
what implies that $u=1$ is a continuity point for both $\sigma^{*} $ and $ \sigma^{**}$. Moreover, $\phi^{n*} (u) \to
\phi^* (u)$ and $\phi^{n**} (u) \to \phi^{**} (u)$ for each continuity point of $\sigma^{*} $ and $ \sigma^{**}$
respectively. Additionally $\phi^{*} $ and $ \phi^{**}$ satisfy the integral equations
\begin{align}
\label{eq: phi^*}
\phi^* (u)=x+\int_{[0;u)}z(\phi^*(s)){\mathrm d} \sigma^*(s),& \quad
\phi^{**} (u)=x+\int_{[0;u)}z(\phi^{**}(s)){\mathrm d} \sigma^{**}(s).
\end{align}
According to lemma's conditions, $\phi^{n}(1)$ converges, therefore $
\phi^{*}(1) = \phi^{**}(1)$.

Let us fix an arbitrary continuity point $u \in [0;1]$ of both
$\sigma^*$ and $\sigma^{**}$ such that $\sigma^*(u) \neq
\sigma^{**}(u)$, suppose that $\sigma^*(u) <
\sigma^{**}(u)$. Since $\sigma^*(u)$ and $\sigma^{**}(u)$
are two different limits of subsequences of $\sigma^n(u)$ then $u$ belongs to set $C$. As it was shown above in this case, sequence $\{n\}$ is a disjoint union $\{n'\}\sqcup\{n''\}$, where $\lim_{n'' \to
\infty}\sigma^{n''}(u)=d$ and $\lim_{n' \to
\infty}\sigma^{n'}(u)=u$. Consequently $\sigma^*(u)=u$ and $\sigma^{**}(u)=d>u$.

 Now consider equation (\ref{eq: phi^*}) with function $z=z_{q,
\varepsilon}$, where $q, \varepsilon \in \mathbb{Q}$ and $u\leq q < q + \varepsilon < d$. Denote their solutions as  $\phi_{q, \varepsilon}^*$ and $\phi_{q,\varepsilon}^{**}$
respectively. Taking $w_{q, \varepsilon}^* = \min\{1, \inf \{t \ge q: \Delta
\sigma^* (t) \ge \varepsilon\}\}$ one can write $\phi_{q, \varepsilon}^*$ in explicit way
$$
\phi^*_{q, \varepsilon} (t) = \left\lbrace
\begin{array}{lcl}
x+\sigma^* (t), \mbox{ if } t \le q ,\\
x + q + \varepsilon - \varepsilon \
{\mathrm e}^{-(\sigma^*(t)-q)/\varepsilon}\prod\limits_{q \le s < t}
(1-\Delta \sigma^* (s)/\varepsilon) {\mathrm e}^{\Delta \sigma^*
(s)/\varepsilon},\\
\mbox{\qquad\qquad\qquad\qquad\qquad\qquad\qquad\qquad\qquad\qquad\quad if }  q < t \le w_{q, \varepsilon}^*,\\
x + q + \varepsilon - \varepsilon \
{\mathrm e}^{-(\sigma^*(w_{q, \varepsilon}^*)-q)/\varepsilon} {\mathrm e}^{-\Delta
\sigma^* (w_{q, \varepsilon}^*)/\varepsilon} \times\\
\qquad\qquad\qquad\quad\ \times\prod\limits_{q \le s \le
w_{q, \varepsilon}^*} (1-\Delta \sigma^* (s)/\varepsilon) {\mathrm e}^{\Delta
\sigma^* (s)/\varepsilon}, \mbox{ if } t > w_{q, \varepsilon}^*.
\end{array}
\right.$$

Letting $\varepsilon \to 0$ in the last formula we obtain $w_{q, \varepsilon}^* \to
w_q^*= \min\{1,\inf \{t \ge q:
\Delta \sigma^* (t) > 0\}\}$ and $w_{q, \varepsilon}^* \ge w_q^*$. Moreover
$$
\phi_{q}^*(1)=\lim_{\varepsilon \to 0}\phi_{q,\varepsilon}^* (1) = \left\{
\begin{array}{lc}
x +\sigma^* (q),\mbox{ if } q < w_q^*\\
x +\sigma^* (w_q^*) + \Delta \sigma^*(w_q^*), \mbox{ if } q = w_q^*.
\end{array}
\right.
$$
Similarly one can obtain
$$
\phi_{u}^*(1)=\lim_{q \to u}\phi_{q}^* (1) = \left\{
\begin{array}{lc}
x +\sigma^* (u),\mbox{ if } u < w_u^*\\
x +\sigma^* (w_u^*) + \Delta \sigma^*(w_u^*), \mbox{ if } u = w_u^*.
\end{array}
\right.
$$
Since $u$ is a continuity point of $\sigma^*$, then $\Delta \sigma^*(w_u^*)=0$ as $u = w_u^*$ and $\phi_{u}^* (1) = x + \sigma^* (u)=x+u$.

Let us consider $\phi_{q, \varepsilon}^{**}$. Taking into account that $\varepsilon \le d-u$, we obtain that $\phi_{q,\varepsilon}^{**}(1) = x + \sigma^{**} (u)=x+d$. Letting $\varepsilon \to 0$, and then $q \to u$ we obtain that $\phi_{u}^{**}(1)=x+d$.

Thus it is shown that $\phi_{u}^* (1) =x+u$ and $\phi_{u}^{**}(1)=x+d$, where $u \neq d$. According to lemma's conditions $\phi_{q,\varepsilon}^{*} (1) = \phi_{q,\varepsilon}^{**} (1)$, therefore $\phi_u^{*} (1) = \phi_u^{**} (1)$, what leads to the contradiction. Hence $\sigma^{n}(u) \to \sigma(u)$ for all continuity points $u \in [0;1]$ of $\sigma$.
\end{proof}
\begin{rem}\label{remark: lemma obratnaya}
Above lemma holds under weaker conditions. Sequence of measures generated by functions $\sigma^n$ converges weakly if number sequence $\phi^n(1) = \varphi_{p_n}^n(z,x)$ converges for countable set of functions $z_{q, \varepsilon}$ of a given type, where $q, \varepsilon \in \mathbb{Q}$.
\end{rem}

Let $j$ be a number such that $t_{j} \le \zeta-1/n < t_{j+1}$.
Set $\xi_k^{n}(t) = F_n (\zeta - t_{j+k})$, $t \in \mathbf{T}$, $n \in \mathbb{N}$ and $k = 0,1,\dots, p+2$, where $p = [1/(n h_n)]$ .
Notice that $\xi_k^{n}(t)$ depends on $t \in \mathbf{T}$, since $t_{j + k}$ depends on $t$.
Denote by $\varphi_k^{n} (z, x, t)$ the sequence which is defined  by the formula (\ref{eq: phi^n}) and
let $\sigma^n_i (u,t)$ be the sequence of functions which is given by the formula (\ref{eq: sigma^n}).

The following lemma from \cite{Y1}  will give the necessary and sufficient conditions for $\sigma^n_i (u,t)$ to satisfy the statement of lemma \ref{mainLemma} and lemma \ref{lemma: obratnaya}.

\begin{lem}[\cite{Y1}]
\label{conditionLemma}
Suppose that $F_n(F_n^{-1}(u)-\delta h_n) \to \sigma(u)$ as $n \to \infty$ and $h_n \to 0$ for all $\delta \in (0;1)$ and all continuity points $u \in [0;1]$ of $\sigma$. Then
$$ \int_{\mathbf{T}} |\sigma^n(u,t)-\sigma(u)|{\mathrm d}t \to 0$$
as $n \to \infty$ and $h_n \to 0$ for all continuity points $u \in [0;1]$ of $\sigma$.

Conversely, if there exists such a function $\sigma(u,t)$, $u \in [0;1]$, $t \in \mathbf{T}$ that $\sigma(u,t) \in \mathbf{L_1 (\mathbf{T})} $ for any $u \in [0;1]$, $\sigma(u,t)$ is left continuous and  nondecreasing on $u$ for any $t \in \mathbf{T}$, and for each continuous function $z: [0;1] \to \mathbb{R}$ we have that
$$\int_{\mathbf{T}} \left| \int_0^1 z(u) \sigma^n ({\mathrm d} u, t) - \int_0^1 z(u) \sigma ({\mathrm d} u, t) \right|{\mathrm d} t \to 0,$$
Then $\sigma ( u, t)$ does not depend on $t$, i.e. $\sigma ( u, t) =\sigma (u)$ and $F_n(F_n^{-1}(u)-\delta h_n) \to \sigma(u)$ for all continuity points $u \in [0;1]$ of $\sigma$ and for any $\delta \in (0;1)$.
\end{lem}

\begin{prop}\label{ineq: phi}
Let $\varphi(z,x,u)$ be a solution of the equation (\ref{eq: phi}) with real-valued function $z$ such that $\left|z(x) - z(y) \right| \leq K_1  \left|x-y\right|$ and $\left|z(x)\right|\leq K_2 (1 + \left|x \right|)$ for all $x,y \in \mathbb{R}$. Then the following inequalities hold for all $x,y \in \mathbb{R}$, $u, v \in [0;1]$, $u<v$:
\begin{align}
& \left | \varphi(z,x,u)-\varphi(z,y,u)\right | \leq \left | x-y\right | \exp(K_1) \nonumber\\
& \left | \varphi(z,x,u) \right | \leq (\left | x\right | + K_2) \exp(K_2) \nonumber \\
& \left | \varphi(z,x,u) - x \right | \leq K_2 (\left | x\right | + 1) \exp(K_2)
\nonumber\\
& \left | \varphi(z,x,u)-x - \varphi(z,y,u) + y \right | \leq \left | x-y\right | K_1 \exp(K_1)    
\nonumber\\
& \left | \varphi(z,x,u)- \varphi(z,x,v)\right | \leq K_2(1+(\left | x\right | + K_2) \exp(K_2))  \mu([u,v))
\nonumber
\end{align}
\end{prop}
\begin{proof}
By using Gronwall inequality, see \cite{G}, these statements are easily shown from the definition of $\varphi(z,x,u)$.
\end{proof}
From now on we will denote by $C$ the constant which depends only on $M$, $K$, $|\mathbf{T}|$ and $V_a^b L$. It does not depend on $n$, $h_n$ and $t \in \mathbf{T}$ and its value can change from one formula to another.
\begin{prop}\label{ineq: x}
Let function $f$ be Lipschitz continuous with constant $M$ satisfying (\ref{eq: fg}). Then for solutions $x$ and $x_n$ of the equations (\ref{eq:limit int}) and (\ref{eq:representetieves system}), respectively, the following inequalities hold for all $t,\ s \in \mathbf{T}$, $t >s $ and $l,\ n \in \mathbb{N}$:
\begin{eqnarray}
\label{solutionInequalityItem1} & & \left | x(t) \right | \leq C (1+\left | x^0 \right |) \\
\label{solutionInequalityItem1a} & & \left | x_n(t) \right | \leq C (1+\left | x^0_n(\tau_t) \right |)\\
\label{solutionInequalityItem2}  & & \left | x(t)-x(s) \right | \leq C(1 + \left | x^0 \right |)V^t_{s}L  \nonumber \\
\label{solutionInequalityItem3} & & \left | x_n(t+l h_n)-x_n(t) \right | \leq C (1 + \left | x^0_n(\tau_t) \right | )V^{t+l h_n + \frac{1}{n}}_tL
\end{eqnarray}
\begin{proof}
The proof of these inequalities is standard and uses the definitions of $x$ and $x_n$, Gronwall inequality, inequality (\ref{eq: fg}), and Lipschitz continuity of $f$.
\end{proof}
\end{prop}

\section{Proof of the main theorem}
In this section we will prove theorem \ref{theorem:necessary}.

Let us recall theorem's conditions. Function $L$ is right-continuous of bounded variation. For any Lipschitz function $f$ that has bounded growth, solution of equation (\ref{eq:representetieves system}) $x_n$ converges in $\mathbf{L_1}(\mathbf{T})$ as $n\rightarrow \infty$, $h_{n}\rightarrow
0$. Let us show that if $L$ has no discontinuity points, then limit of $x_n$ is a solution of integral equation (\ref{eq:limit int}); and if $L$ has at least one point of discontinuity, then there exists function $\sigma \in G$ such that $F_{n}(F^{-1}_{n}(u)-\delta h_{n})\rightarrow \sigma(u)$ as
$n\rightarrow \infty$, $h_{n}\rightarrow 0$, for all $\delta\in
(0,1)$ and for all continuity points $ u \in [0,1]$ of $\sigma$. Moreover limit of $x_n$ is a solution of integral equation (\ref{eq:limit int}) with measure $\mu$ generated by $\sigma$.

\begin{prop}\label{statement:x0}
In terms of theorem \ref{theorem:necessary} there exists $x_0 \in \mathbb{R}$ such that
\begin{equation}
\int_{\mathbf{T}}\left | x^0_n(\tau_{t})-x^0 \right | {\mathrm d} t \to 0.
\end{equation}
\end{prop}
\begin{proof}
As it was shown above, each point $t$ can be represented as $t=\tau_{t}+m_{t}h_{n}$, where $\tau_{t}\in [a, a+h_{n}]$,  $m_{t}\in \mathbb{N}$ and solution of equation (\ref{eq:representetieves system}) can be written as follows:
\begin{equation}\tag{\ref{eq:representetieves explicit}}
x_{n}(t)=x_{n}(\tau_{t})+\sum^{m_{t}-1}_{k=0}f_{n}(t_{k},x_{n}(t_{k}))[L_{n}(t_{k+1})-L_{n}(t_{k})].
\end{equation}
Set $f\equiv0$, then $ x_n(t)= x_n(\tau_{t}) = x^0_n(\tau_{t}).$
According to conditions of theorem (\ref{theorem:necessary}), $x_n$ converges in $\mathbf{L_1(\mathbf{T})}$, therefore $x^0_n(\tau_{t})$ converges. Taking an arbitrary function $g \in \mathbf{C(\mathbf{T})}$ consider the following limit:
$$
I=\lim_{n \to \infty}\int^b_a g(t) x^0_n(\tau_{t}) {\mathrm d} t = \lim_{n \to \infty} \sum^{[\frac{b-a}{h_n}]}_{k=1} \int^{a + k h_n}_{a + (k-1) h_n} g(t) x^0_n(a + h_n \{\frac{t-a}{h_n}\}) {\mathrm d} t.
$$
Directly from the definition of fractional part we obtain
\begin{eqnarray*}
I &= \lim\limits_{n \to \infty} \int\limits^1_0 \sum\limits^{[\frac{b-a}{h_n}]}_{k=1} g(h_n s + a + (k-1) h_n) h_n x^0_n(a + h_n s) {\mathrm d} s.
\end{eqnarray*}
Since function $g$ is bounded on $\mathbf{T}$ then by Lebesgue dominated convergence theorem we have
\begin{eqnarray*}
I &=& \lim_{n \to \infty} \int^1_0 \int^b_a g(t) {\mathrm d} t \hspace{5pt}x^0_n(a + h_n s) {\mathrm d} s = \int^b_a g(t) {\mathrm d} t\lim_{n \to \infty} \int^1_0 x^0_n(a + h_n s) {\mathrm d} s .
\end{eqnarray*}
Denote $x^0 = \lim_{n \to \infty} \int^1_0 x^0_n(a + h_n s) {\mathrm d} s$, then
\begin{eqnarray*}
I &=& \int^b_a g(t) x^0 {\mathrm d} t.
\end{eqnarray*}

Since $x^0_n(\tau_{t})$ converges in $\mathbf{L_1(\mathbf{T})}$ then $x^0_n(\tau_{t}) \to x^0$ in $\mathbf{L_1(\mathbf{T})}$.
\end{proof}

\begin{prop}\label{statement:L^c}
Theorem \ref{theorem:necessary} holds for continuous function $L$.
\end{prop}
\begin{proof}

Let $y$ be a solution of integral equation (\ref{eq:limit int}). Since $L$ is continuous, then $L^c=L$ and equation (\ref{eq:limit int}) can be represented as follows:
$$
y(t)=x^{0}+\int^{t}_{a}f(s,y(s)){\mathrm d}L^{c}(s),
$$
where $x^{0}$ -- limit of initial conditions $x^0_n(\tau_{t})$. According to the proposition \ref{statement:x0}, value $x^0$ exists and it is unique.

Let us show that $\int^b_a \left | x_n(t) - y(t) \right | {\mathrm d} t \to 0$ as $n \to \infty$. Then due to the limit uniqueness, $x$ will have a required format.

Taking into account explicit form of $x_n$ and $y$, inequalities (\ref{solutionInequalityItem1}), (\ref{solutionInequalityItem1a}) and also properties of $f_n$ and $f$ with the help of standard methods one can show that
\begin{equation*}
\begin{split}
&\left | x_n(t) - y(t) \right | \leq \left | x^0_n(\tau_t) - x^0 \right | + \\
&\qquad+ \left | \sum_{k=0}^{m_t-1}f_n(t_k,x_n(t_k))[L_n^c(t_{k+1})-L_n^c(t_{k})]-\int_{a}^{t}f(s,y(s)){\mathrm d}L^c(s) \right |\leq\\
&\leq \left | x^0_n(\tau_t) - x^0 \right | + C/n +C(1+\left | x_n^0(\tau_t)\right | + \left | x^0 \right |)\sup_{\left | u-v  \right |\leq h_n+1/n}V_u^vL^c+\\
& \qquad + C\sum_{k=0}^{m_t-1}\left | x_n(t_k)-y(t_k) \right |[L^c(t_{k+1})-L^c(t_k)].
\end{split}
\end{equation*}
Applying Gronwall inequality to the above expression we obtain that:
\begin{equation*}
\begin{split}
\left | x_n(t) - y(t) \right | &\leq [\left | x^0_n(\tau_t) - x^0 \right | + C/n +\\
&+C(1+\left | x_n^0(\tau_t)\right | + \left | x^0 \right |)\sup_{\left | u-v  \right |\leq h_n+1/n}V_u^vL^c] \exp(C V^b_aL^c).
\end{split}
\end{equation*}
Taking integral over $\mathbf{T}$ and considering the limit as $n \to \infty$ we obtain that $$\int^b_a \left | x_n(t) - y(t) \right | {\mathrm d} t \to 0.$$
Thus theorem \ref{theorem:necessary} holds for continuous function $L$.
\end{proof}

Now let us show that theorem \ref{theorem:necessary} holds for discontinuous function $L$.
In order to emphasise that solution of equation (\ref{eq:representetieves system}) depends on function $f$, we will denote it as $x_n = x_n (f)$, and its value at point $t$ as $x_n(t) = x_n(f,t)$.
Thus we need to prove that if for any Lipschitz function $f$ of a bounded growth sequence $x_n (f)$ of solutions of equation
(\ref{eq:representetieves system}) converges in $\mathbf{L_1(T)}$, then
$F_{n}(F^{-1}_{n}(u)-\delta h_{n})$ converges for all $\delta\in
(0,1)$ and almost all (in general, for all except for no more than countable number of points) $u \in [0,1]$.



For any $t\in \mathbf{T}$ define $t_k=t_k(t)=\tau_t+k h_n$, $k=0,1,2,\ldots$ as above. Pick any point $\zeta \in \mathbf{T}$. Let $q \in \mathbb{N}$ be a number such that
\begin{equation}\label{t_q}
t_{q-1}(t)< \zeta \le t_q(t).
\end{equation}
Notice that $t_q(t+h_n)=t_q(t)$, i.e. $t_q(t)$ is a periodical function with period $h_n$.

\begin{prop}\label{prop:x^+}
For any Lipschitz continuous function $f$ of a bounded growth and for each $\zeta \in \mathbf{T}$ there exists $x^+(\zeta) \in \mathbb{R}$ such that
\begin{equation}\label{x^+}
\lim_{n\to \infty}\int_{\zeta}^b |x_n(t_q (t)) -
x^+(\zeta)| d t = 0.
\end{equation}
\end{prop}
\begin{proof}


For each $\delta \in
(0;1)$ we denote by $g(\delta,\zeta,t)$ the following function
\begin{equation*}
g(\delta,\zeta,t) = \left\{ \begin{array}{ll}
 1, \mbox{ if } t \le \zeta + \delta/2, \\
2 (\delta + \zeta - t)/\delta, \mbox{ if }  \zeta + \delta/2 < t \le \zeta + \delta,\\
 0, \mbox{ if } t > \zeta + \delta.
\end{array} \right.
\end{equation*}

Set $f^\delta (t,x) = f(t,x) g(\delta,\zeta,t) $, then
\begin{equation*}
f^\delta (t,x) = \left\{ \begin{array}{ll}
                          f(t,x), \mbox{ if }t \le \zeta + \delta/2 \\
              0, \mbox{ if }t \ge \zeta + \delta.
                         \end{array}\right.
\end{equation*}
Moreover, function $f^\delta$ is Lipschitz and it satisfies (\ref{eq: fg}) with the same constant as $f$ does.

Let $x_n^\delta = x_n(f^\delta)$ and $x_n = x_n(f)$ be a solutions of equation (\ref{eq:representetieves system}) with functions $f^\delta$ and $f$ respectively.

Note that  $x_n^\delta (t) = x_n(t)$ if $t < \zeta
+\delta/2 - 1/n$ and $x_n^\delta (t_q(t)) = x_n(t_q(t))$ for large enough $n$ and $t \in \mathbf{T}$.
Moreover, $x_n^\delta (t) = c_n^\delta(t)$
if $t > \zeta +\delta$ where $c_n^\delta$ is some $h_n$-periodic function.
Now $x_n^\delta$ converges to $x^\delta$ in
$\mathbf{L_1(T)}$, therefore sequence of functions $c_n^\delta$ also converges in $\mathbf{L_1([\zeta +\delta,b])}$.
Denote limit of $c_n^\delta$ by $c^\delta$. The same arguments as in proposition \ref{statement:x0} imply that function $c^\delta$ does not depend on $t$, i.e. it is a constant.

Let us show that number sequence $c^\delta$ converges as $\delta \to 0$. Consider arbitrary numbers $\delta_1, \delta_2 \in (0,1)$ and correspondent functions $f^{\delta_1}$,
$f^{\delta_2}$. Notice that $f_n^{\delta_1}(t,x)=f_n^{\delta_2}(t,x)$ only if $t \not\in (\zeta
+\min\{\delta_1,\delta_2\}/2 - 1/n, \zeta + \max\{\delta_1,\delta_2\}]$. Let $r$ and $r^*$ be numbers such that $t_{r-1}\leq \zeta
+\min\{\delta_1,\delta_2\}/2 - 1/n <t_r \leq t_{r^*}\leq \zeta
+\max\{\delta_1,\delta_2\}<t_{r^*+1}$. Then from equality (\ref{eq:representetieves explicit}) we get
\begin{equation*}
\begin{split}
\left | x_n^{\delta_1}(t) - x_n^{\delta_2}(t) \right | &= \sum^{m_t-1}_{k=0}[f^{\delta_1}_n(t_k,x_n^{\delta_1}(t_k))-f^{\delta_2}_n(t_k,x_n^{\delta_2}(t_k))][L_n(t_{k+1})-L_n(t_k)]= \\
&=\sum^{r^*}_{k=r}[f^{\delta_1}_n(t_k,x_n^{\delta_1}(t_k))-f^{\delta_2}_n(t_k,x_n^{\delta_2}(t_k))][L_n(t_{k+1})-L_n(t_k)].
\end{split}
\end{equation*}
Using inequality (\ref{solutionInequalityItem3}) by standard methods one can show that
\begin{equation*}
\begin{split}
\left | x_n^{\delta_1}(t) - x_n^{\delta_2}(t) \right | &\leq C (1+\left | x_n(t_r)\right |)V_{\zeta - 1/n +\min\{\delta_1,\delta_2\}/2}^{\zeta + \max\{\delta_1,\delta_2\} +h_n}L_n \leq\\
& \leq C (1+\left | x_n^0(\tau_t)\right |)V_{\zeta-1/n+\min\{\delta_1,\delta_2\}/2}^{\zeta +
\max\{\delta_1,\delta_2\} + h_n + 1/n}L.
\end{split}
\end{equation*}
Taking integral over interval $[\zeta
+\max\{\delta_1,\delta_2\}, b]$ and tending $n \to
\infty$ we obtain that
\begin{equation*}
\begin{split}
&\left | c^{\delta_1} - c^{\delta_2}\right | (b - \zeta -
\max\{\delta_1,\delta_2\}) = \lim_{n \to \infty}\int_{\zeta
+\max\{\delta_1,\delta_2\}}^{b}\left | x_n^{\delta_1}(t) -
x_n^{\delta_2}(t) \right | {\mathrm d} t\leq\\
&\leq \lim_{n \to \infty} C \int_{\zeta
+\max\{\delta_1,\delta_2\}}^{b}(1+\left | x_n^0(\tau_t)\right |){\mathrm d} t
 \hspace{4pt} V_{\zeta-1/n+\min\{\delta_1,\delta_2\}/2}^{\zeta +
\max\{\delta_1,\delta_2\} + h_n + 1/n}L  \leq\\
&\leq C (1+\left | x^0\right |) V_{\zeta+\min\{\delta_1,\delta_2\}/2}^{\zeta +
\max\{\delta_1,\delta_2\}}L.
\end{split}
\end{equation*}
Since $L$ is a right-continuous function, above inequality implies that $c^\delta$ is a Cauchy sequence. Thus there exists $\lim_{\delta \to
0}c^\delta = x^+(\zeta)$.

Now let us recall that $c^\delta_n$ is $h_n$-periodic function and $x_n^\delta (t_q(t)) = x_n(t_q(t))$ for large enough $n$. Then
$$ |x_n(t_q (t)) - c^\delta_n(t_q (t))| = |x^\delta_n(t_q (t)) - c^\delta_n(t)|.$$
Since for $t > \zeta + \delta$ we have that $f^\delta(t,x)=0$ and $x^\delta_n(t)=c^\delta_n(t)$ then using inequality (\ref{solutionInequalityItem1a}) we deduce that
\begin{equation*}
\begin{split}
| x_n^\delta (t_q (t)) - c^\delta_n(t)| &= | x_n^\delta (t_q (t)) - x^\delta_n(t)|
\le C (1+\left | x^\delta_n(t_q (t))\right |)V_{\zeta}^{\zeta + \delta + h_n
+ 1/n} L \leq\\
&\leq C (1+\left | x^0_n(\tau_t)\right |)V_{\zeta}^{\zeta + \delta + h_n
+ 1/n} L.
\end{split}
\end{equation*}
Thus we obtain the following inequality
\begin{equation*}
\begin{split}
|x_n(t_q (t)) - x^+(\zeta)| &\le | x_n^\delta (t_q (t)) - c^\delta_n(t)|
+ |c^\delta_n(t) - x^+(\zeta)| \le\\
&\le C (1+\left | x_n^0(\tau_t)\right |)V_{\zeta}^{\zeta + \delta + h_n
+ 1/n} L + |c^\delta_n(t) - x^+(\zeta)|.
\end{split}
\end{equation*}
Taking integral over interval $[\zeta+\delta,
b]$  and letting $n \to \infty$ we have

\begin{equation*}
\begin{split}
&\overline{\lim}_{n\to \infty}\int_{\zeta+\delta}^b|x_n(t_q (t)) - x^+(\zeta)| {\mathrm d} t \le\\
&\le C (1+\left | x^0\right |)V_{\zeta}^{\zeta + \delta} L + \lim_{n\to \infty}\int_{\zeta+\delta}^b|c^\delta_n(t) - c^\delta|{\mathrm d} t +|c^\delta - x^+(\zeta)|(b-\zeta-\delta)=\\
&= C (1+\left | x^0\right |)V_{\zeta}^{\zeta + \delta} L +|c^\delta - x^+(\zeta)|(b-\zeta-\delta).
\end{split}
\end{equation*}
Tending $\delta \to 0$ we obtain required equality
\begin{equation*}
\lim_{n\to \infty}\int_{\zeta}^b |x_n(t_q (t)) -
x^+(\zeta)| d t = 0.
\end{equation*}
\end{proof}

Now for $t \in \mathbf{T}$ and $m \in \mathbb{N}$ let $t_j(t,m)=\tau_t + jh_n$ be such a point that
\begin{equation}\label{t_j}
t_{j}(t,m) < \zeta - \frac{1}{m} \le t_{j+1}(t,m).
\end{equation}
Notice that $t_{j}(t+h_n,m)=t_{j}(t,m)$ for all $t \in \mathbf{T}$.

\begin{prop}\label{prop:x^-}
For any Lipschitz continuous function $f$ of a bounded growth and for each $\zeta \in \mathbf{T}$ there exists $x^-(\zeta)$ such that
\begin{equation}\label{x^-}
\lim_{n\to \infty}\int_{\zeta}^b |x_n(t_{j} (t,n)) -x^-(\zeta)| d t = 0.
\end{equation}
\end{prop}
\begin{proof}


Applying equality (\ref{x^+}) at point $\zeta - \frac{1}{m}$ we have
\begin{equation}\label{x_m^+}
\lim_{n\to \infty}\int_{\zeta}^b |x_n(t_{j+1} (t,m)) -
x^+(\zeta - \frac{1}{m})| d t = 0.
\end{equation}

Now we show that sequence $x^+(\zeta - \frac{1}{m})$ converges as $m \to \infty$. It is sufficient to prove that $\left | x^+(\zeta - \frac{1}{m}) - x^+(\zeta - \frac{1}{k}) \right | \to 0$ as $m,k \to \infty$.

Introduce $k,m \in \mathbb{N}$ such that $k < m$ and $\delta < \frac{1}{m k}$. Denote $f^{\delta,m}(t,x)= f(t,x) g(\delta, \zeta - \frac{1}{m}, t)$ and $x_n^{\delta,m} = x_n(f^{\delta,m})$ -- solution of equation (\ref{eq:representetieves system}) with function $f^{\delta,m}$. Due to special form of $\delta$  we have that  $x_n^{\delta,m} (t) = x_n^{\delta,k} (t)$ if $t < \zeta - 1/k - 1/n$. Moreover, as it was shown above, there exists number $c^\delta(\zeta-\frac{1}{m})$ such that $$\lim_{n \to \infty}\int_\zeta^b \left | x_n^{\delta,m}(t) - c^\delta(\zeta-\frac{1}{m})\right | {\mathrm d} t = 0.$$
It is easy to see that
$$\left | c^\delta(\zeta-\frac{1}{k})- c^\delta(\zeta-\frac{1}{m})\right | = \frac{1}{b-\zeta} \lim_{n \to \infty}\int_\zeta^b \left | x_n^{\delta,m}(t) -x_n^{\delta,k}(t) \right | {\mathrm d} t.$$
Equality (\ref{eq:representetieves explicit}) implies
$$\left | x_n^{\delta,m}(t) -x_n^{\delta,k}(t) \right | = \left | \sum^{m_t-1}_{i=0}[f^{\delta,m}_n(t_i,x_n^{\delta,m}(t_i))-f^{\delta,k}_n(t_i,x_n^{\delta,k}(t_i))][L_n(t_{i+1})-L_n(t_i)] \right |.$$
Since $\delta < \frac{1}{mk}$, $f^{\delta, m}_n = f^{\delta, k}_n$ if $t \not \in (\zeta - \frac{1}{k}-\frac{1}{n}; \zeta - \frac{1}{m}+\frac{1}{mk})$. Denote by $t_r$ the first point that belongs to this interval and $t_{r^*}$ -- the last one. Then using inequality (\ref{solutionInequalityItem3}) and the same technique as above we have
\begin{equation*}
\begin{split}
&\left | x_n^{\delta,m}(t) -x_n^{\delta,k}(t) \right | =\\
&= \left | \sum^{r^*}_{i=r}[f^{\delta,m}_n(t_i,x_n^{\delta,m}(t_i))-f^{\delta,k}_n(t_i,x_n^{\delta,k}(t_i))][L_n(t_{i+1})-L_n(t_i)] \right | \leq\\
&\leq C(1+\left | x^0_n(\tau_t)\right |)V_{\zeta - \frac{1}{k}-\frac{1}{n}}^{\zeta - \frac{k-1}{mk}+h_n + \frac{1}{n}}L.
\end{split}
\end{equation*}
After taking integral over  $[\zeta
, b]$ and letting $n \to
\infty$ we obtain that
\begin{equation*}
\begin{split}
\left | c^\delta(\zeta-\frac{1}{k})- c^\delta(\zeta-\frac{1}{m})\right | &= \frac{1}{b-\zeta} \lim_{n \to \infty}\int_\zeta^b \left | x_n^{\delta,m}(t) -x_n^{\delta,k}(t) \right | {\mathrm d} t \leq\\
&\leq C(1+\left | x^0 \right |)V_{\zeta - \frac{1}{k}}^{\zeta - \frac{k-1}{mk}}L.
\end{split}
\end{equation*}
Recall that $c^\delta(u) \to x^+(u)$ as $\delta \to 0$. Therefore the above inequality implies
\begin{equation*}
\begin{split}
&\left | x^+(\zeta - \frac{1}{m}) - x^+(\zeta - \frac{1}{k}) \right | \leq C(1+\left | x^0 \right |)V_{\zeta - \frac{1}{k}}^{\zeta - \frac{k-1}{mk}}L.
\end{split}
\end{equation*}
Thus $\left | x^+(\zeta - \frac{1}{m}) - x^+(\zeta - \frac{1}{k}) \right | \to 0$ as $m,k \to \infty$. Therefore sequence $x^+(\zeta - \frac{1}{k})$ converges.
Denote by $x^-(\zeta)$ its limit.

From equality (\ref{x_m^+}) we have
$$\lim_{m\to \infty}\lim_{n\to \infty}\int_{\zeta}^b |x_n(t_{j+1} (t,m)) -
x^-(\zeta)| d t = 0.$$

Now we calculate $\lim_{n\to \infty}\int_{\zeta}^b |x_n(t_{j} (t,n)) -x^-(\zeta)| d t$. It is evident that for each $m$ and large enough $n$ for all $t \in \mathbf{T}$ $t_{j+1} (t,m)< t_{j} (t,n)$.

Equation (\ref{eq:representetieves explicit}) yields
\begin{equation*}
\begin{split}
&\lim_{n\to \infty}\int_{\zeta}^b |x_n(t_{j} (t,n)) -x^-(\zeta)| d t \leq\\
&\leq\lim_{m\to \infty}\lim_{n\to \infty}\int_{\zeta}^b |x_n(t_{j} (t,n)) - x_n(t_{j+1} (t,m))| d t +\\
&\quad+ \lim_{m\to \infty}\lim_{n\to \infty}\int_{\zeta}^b |x_n(t_{j+1} (t,m)) -x^-(\zeta)| d t \leq\\
&\leq\lim_{m\to \infty}\lim_{n\to \infty}\int_{\zeta}^b C (1 + \left | x^0_n(\tau_t) \right | )V^{t_{j}(t,n) + \frac{1}{n}}_{t_{j+1} (t,m)}L d t \leq \lim_{m\to \infty} C (1 + \left | x^0 \right | )V^{\zeta -}_{\zeta - \frac{1}{m}}L =0,
\end{split}
\end{equation*}
since $L$ -- a right-continuous function.

Thus
\begin{equation*}
\lim_{n\to \infty}\int_{\zeta}^b |x_n(t_{j} (t,n)) -x^-(\zeta)| d t = 0.
\end{equation*}
\end{proof}

\begin{prop}\label{prop:phi_n->x+}
Let $\zeta \in \mathbf{T}$ be a point of discontinuity of $L$, then
\begin{equation} \label{eq: phi_n->x+}
\varphi^{n}_{p+2}(\Delta L(\zeta) f(\zeta,\cdot),x^- (\zeta),t)
\to x^+ (\zeta)  \mbox{ in $\mathbf{L_1(T)}$ as } n \to \infty,
\end{equation}
where $x^+ (\zeta)$ and $x^- (\zeta)$ are defined in propositions \ref{prop:x^+} and \ref{prop:x^-}, $p=[1/(n h_n)]$ and function $\varphi^{n}_k(z,x,t)$ is defined by (\ref{eq: varphi}) with $\xi^n_k (t) = F_n (\zeta - t_{j+k}(t))$.
\end{prop}
\begin{proof}
Recall that $t_q=t_q(t)$ and $t_j=t_j(t,n)$ are defined by inequalities (\ref{t_q}) and (\ref{t_j}) respectively.
As $p=[1/(n h_n)]$, then
$$
x_n(t_{j+p+2}) - x_n (t_j) = \sum_{i=0}^{p+1}
f_n(t_{j+i},x_n(t_{j+i})) (L_n (t_{j+i+1}) - L_n (t_{j+i})).
$$
It is easy to see that $t_{j+p+2}$ is equal to $t_q$ or $t_{q+1}$. Therefore equalities (\ref{x^+}) and (\ref{x^-}) imply that
\begin{equation} \label{eq: f_n(x_n)L_n}
I_n = \sum_{i=0}^{p+1} f_n(t_{j+i},x_n(t_{j+i})) (L_n (t_{j+i+1})
- L_n (t_{j+i})) \xrightarrow[{n \to \infty }]{} x^+ (\zeta) - x^- (\zeta) \mbox{ in $\mathbf{L_1(T)}$}.
\end{equation}

Now let us represent function $L$ in the following form: $L=L^*+L^d$ where $$L^d(t)=\left\{ \begin{array}{ll}
 0, \mbox{ if } t < \zeta, \\
\Delta L (\zeta), \mbox{ if }  t \ge \zeta.
\end{array} \right.$$
Then we have
\begin{equation*}
\begin{split}
&L^d_n(t_{j + i+1}) - L^d_n(t_{j +i}) = \int_0^{1/n} (L^d(t_{j + i+1}+s) - L^d(t_{j + i}+s)) \rho _n (s) {\mathrm d}s =\\
&=\int^{\zeta - t_{j+i+1}}_{\zeta - t_{j+i+1} - 1/n} (L^d(\zeta-u) - L^d(\zeta - h_n -u)) \rho _n (\zeta - t_{j+i+1}-u) {\mathrm d}u=J.\\
\end{split}
\end{equation*}
From definition of $L^d$ one can see that $L^d(\zeta-u) - L^d(\zeta - h_n -u) = 0$ if $u \not \in [-h_n;0)$. Then since $\mathrm{supp} \ \rho_n \subset [0; 1/n]$ we have
\begin{equation*}
\begin{split}
J&=\int^{0}_{-h_n} (L^d(\zeta-u) - L^d(\zeta - h_n -u)) \rho _n (\zeta - t_{j+i+1}-u) {\mathrm d}u=\\
&=\Delta L(\zeta) \int^{0}_{-h_n} \rho _n (\zeta - t_{j+i+1}-u) {\mathrm d}u=\Delta L(\zeta) \int^{\zeta - t_{j+i}}_{\zeta - t_{j+i+1}} \rho _n (s) {\mathrm d}s=\\
&=\Delta L(\zeta)[F_n(\zeta - t_{j + i + 1})-  F_n(\zeta - t_{j + i})]=\Delta L(\zeta)[\xi^{n}_{i+1}(t)-\xi^{n}_i(t)].
\end{split}
\end{equation*}


Applying methods described in lemma 5.6 from \cite{YN} and using an explicit form of $f_n$, $L_n$ and $\varphi^{n}_{p+2}$ one can show that
\begin{equation*}
\begin{split}
&\left | \sum_{i=0}^{p+1} f_n(t_{j+i},x_n(t_{j+i})) (L_n (t_{j+i+1})
- L_n (t_{j+i})) - \right.\\
&\qquad\qquad\qquad\qquad\qquad\qquad\left. \phantom{\sum_i^i} - (\varphi^{n}_{p+2}(\Delta L(\zeta)
f(\zeta,\cdot),x_n (t_j),t) - x_n (t_j) ) \right | \le\\
&\le C(1+|x_n^0 (\tau_t)|)( 1/n +  h_n + V_{\zeta-1/n}^{\zeta -} L
+ V_{\zeta}^{\zeta +1/n+h_n} L).
\end{split}
\end{equation*}

Therefore formula (\ref{eq: f_n(x_n)L_n}) yields
$$
\varphi^{n}_{p+2}(\Delta L(\zeta) f(\zeta,\cdot),x_n (t_j),t) -
x_n (t_j) \to x^+ (\zeta) - x^- (\zeta) \mbox{ in $\mathbf{L_1(T)}$}.
$$
Using proposition \ref{ineq: phi} we obtain that
\begin{equation*}
\begin{split}
&\left|\varphi^{n}_{p+2}(\Delta L(\zeta) f(\zeta,\cdot),x_n (t_j),t) - x_n (t_j) - \varphi^{n}_{p+2}(\Delta L(\zeta) f(\zeta,\cdot),x^- (\zeta),t) + x^- (\zeta) \right| \leq\\
&\leq C \left|x_n (t_j) - x^- (\zeta)\right|.
\end{split}
\end{equation*}

Since from (\ref{x^-}) we have that $x_n (t_j) \to x^- (\zeta)$  in $\mathbf{L_1(T)}$, then
\begin{equation*}
\varphi^{n}_{p+2}(\Delta L(\zeta) f(\zeta,\cdot),x^- (\zeta),t)
\to x^+ (\zeta) \mbox{ in $\mathbf{L_1(T)}$}.
\end{equation*}
\end{proof}
%

\begin{prop}\label{prop:phi_x0}
Sequence of functions $\varphi^{n}_{p+2}(\Delta L(\zeta) f(\zeta,\cdot),x^0,t)$ converges in $\mathbf{L_1(T)}$.
\end{prop}
\begin{proof}
 Consider the following set of functions:
\begin{equation*}
w(\delta,\zeta,t) = \left\{ \begin{array}{ll}
 0, \mbox{ if } t \le \zeta - \delta, \\
(t + \delta - \zeta )/\delta, \mbox{ if }  \zeta - \delta < t \le \zeta,\\
 1, \mbox{ if } t > \zeta.
\end{array} \right.
\end{equation*}
Denote by $x_n^\delta$ solution of equation (\ref{eq:representetieves system}) with function
 $g_\delta(t,x)=f(t,x) w(\delta,\zeta,t)$.
As it was shown above, $x_n^\delta (t_j) \to x^-(\zeta, g_\delta)$.  Moreover, using equality (\ref{eq:representetieves explicit}) one can show that
$$
|x_n^\delta (t_j) - x_n^0(\tau_t)| \le C (1+|x_n^0(\tau_t)|)
V_{\zeta - \delta}^{t_j + 1/n} L.
$$
Taking integral over $\mathbf{T}$ and letting $n\to \infty$ we get
$$
|x^- (\zeta, g_\delta) - x^0| \le C (1+|x^0|) V_{\zeta -
\delta}^{\zeta - } L.
$$
Thus $x^- (\zeta, g_\delta)$ converges to $x^0$ as $\delta \to 0$.

Let us show that $x^+(\zeta, g_\delta)$ also converges. Since $\Delta L(\zeta)f(\zeta,\cdot)=\Delta L(\zeta)g_\delta(\zeta,\cdot)$ formula (\ref{eq: phi_n->x+}) implies that $\varphi^{n}_{p+2}(\Delta L(\zeta) f(\zeta,\cdot),x^- (\zeta,g_\delta),t) \to x^+(\zeta, g_\delta)$. Using definition of $\varphi_n$ we obtain the following inequality:
\begin{equation*}
\begin{split}
&\left |\varphi^{n}_{p+2}(\Delta L(\zeta) f(\zeta,\cdot),x^- (\zeta,g_{\delta_1}),t) - \varphi^{n}_{p+2}(\Delta L(\zeta) f(\zeta,\cdot),x^- (\zeta, g_{\delta_2}),t)\right | \le \\
&\le C \left |x^-(\zeta, g_{\delta_1}) - x^- (\zeta,g_{\delta_2})\right |.
\end{split}
\end{equation*}
Therefore
\begin{equation}\label{x^+->x0}
|x^+(\zeta, g_{\delta_1}) - x^+(\zeta, g_{\delta_2})| \le C |x^- (\zeta,g_{\delta_1}) - x^-
(\zeta,g_{\delta_2})|.
\end{equation}
Since sequence $x^-(\zeta, g_{\delta})$ converges, then from (\ref{x^+->x0}) we obtain that sequence $x^+(\zeta, g_{\delta})$ also converges. Denote its limit by $x^+(\zeta, f^0 )$.

Definition of $\varphi_n$ implies
\begin{equation*}
\begin{split}
&|\varphi^{n}_{p+2}(\Delta L(\zeta) f(\zeta,\cdot),x^0 ,t) -
x^+(\zeta, f^0 )| \le\\
&\le |\varphi^{n}_{p+2}(\Delta L(\zeta)
f(\zeta,\cdot),x^0 ,t)  - \varphi^{n}_{p+2}(\Delta L(\zeta)
f(\zeta,\cdot),x^- (\zeta,g_{\delta}),t)| +\\
&\quad+| \varphi^{n}_{p+2}(\Delta L(\zeta) f(\zeta,\cdot),x^- (\zeta,g_{\delta}),t) - x^+(\zeta, g_{\delta})| +
|x^+(\zeta, g_{\delta}) - x^+(\zeta, f^0 )| \le\\
&\le C |x^- (\zeta,g_{\delta}) - x^0| +|
\varphi^{n}_{p+2}(\Delta L(\zeta) f(\zeta,\cdot),x^- (\zeta,g_{\delta}),t) - x^+(\zeta, g_{\delta})| +\\
&\quad +|x^+(\zeta, g_{\delta}) - x^+(\zeta, f^0 )|.
\end{split}
\end{equation*}
Taking integral over interval $[\zeta;b]$ and letting $n \to \infty$ we have
\begin{equation*}
\begin{split}
\lim_{n \to \infty} \int_{\zeta}^b |\varphi^{n}_{p+2}(\Delta
L(\zeta) f(\zeta,\cdot),x^0 ,t) - x^+(\zeta, f^0 )|dt &\le C |x^-
(\zeta,g_{\delta}) - x^0| +\\
&+ C |x^+(\zeta, g_{\delta}) - x^+(\zeta, f^0 )|.
\end{split}
\end{equation*}
Tending $\delta \to 0$ we obtain that
$\varphi^{n}_{p+2}(\Delta L(\zeta) f(\zeta,\cdot),x^0 ,t)$
converges in $\mathbf{L_1(T)}$.
\end{proof}

\begin{pr} \end{pr} 

Proposition \ref{statement:L^c} states that theorem \ref{theorem:necessary} holds for continuos function $L$.

For discontinus function $L$ let us introduce some notations. $$\varphi^{n}_{p+2}(\Delta L(\zeta) f(\zeta,\cdot),x^0 ,t) = \phi^n(1,t),$$
where $\varphi^{n}_{p+2}$ and $\xi^n_k (t)$ are defined in proposition \ref{prop:phi_n->x+}. Then function $\sigma^n(s,t)$ is generated by $\xi^n_k (t)$ according to (\ref{eq: sigma^n}) and $\phi^n$ is a solution of the corresponding integral equation:
$$
\phi^n(u,t)=x^0+\int_{[0;u)}\Delta L(\zeta) f(\zeta,\phi^n(s,t))\sigma^n(ds,t).
$$
Proposition \ref{prop:phi_x0} implies that there exists such a number $x^0 \in \mathbb{R}$ that for any Lipshitz function $f$ of a bounded growth and for any discontinuity point $\zeta$ of $L$ sequence $\phi^n(1,t)$ converges in $\mathbf{L_1(\mathbf{T})}$.
Since for each $f$ sequence $\phi^n(1,t)$ converges in $\mathbf{L_1(\mathbf{T})}$, there exists such a subsequence $n_k$, depending on $f$, that $\phi^{n_k}(1,t)$ converges almost everywhere.

Fix $t \in \mathbf{T}$ in such a way that $\phi^{n_k}(1,t)$ converges for a countable number of functions $f=z_{q,\varepsilon}$, $q, \varepsilon \in \mathbb{Q}$ defined in lemma \ref{lemma: obratnaya}. Denote $\phi^{n_k}(1)=\phi^{n_k}(1,t)$ and $\sigma^{n_k}(u)=\sigma^{n_k}(u,t)$.
Thus, we figured out that there exists a sequence $n_k$, for which sequence $\phi^{n_k}(1)$ converges for any functions $z_{q,\varepsilon}$. From remark \ref{remark: lemma obratnaya} we deduce that lemma \ref{lemma: obratnaya} holds therefore.
Lemma \ref{lemma: obratnaya} in its turn implies that sequence of functions $\sigma^{n_k}$ converges to $\sigma \in G$ for any continuity point $u$ of $\sigma$. This means that lemma \ref{conditionLemma} conditions are also held, so we  obtain that
\begin{equation}\label{eq: F n_k}
F_{n_k}(F^{-1}_{n_k}(u)-\delta h_{n_k}) \to \sigma(u),
\end{equation}
as $n_k\rightarrow \infty$, $h_{n_k}\rightarrow 0$ for all $\delta\in
(0,1)$ and for each continuity point $u$ of $\sigma$.

It is easy to see that for each subsequence $n'$ one can select subsequence $n_k'$ such that convergence (\ref{eq: F n_k}) holds. Therefore (\ref{eq: F n_k}) holds for the whole sequence $n$, i.e.
\begin{equation}\label{eq: F n}
F_{n}(F^{-1}_{n}(u)-\delta h_{n}) \to \sigma(u)
\end{equation}
as
$n\rightarrow \infty$, $h_{n}\rightarrow 0$, for all $\delta\in
(0,1)$ and for all continuity points $u \in [0,1]$ of
$\sigma$.

Formula (\ref{eq: F n}) means that conditions of theorem \ref{theorem:suficient} hold. Applying theorem \ref{theorem:suficient} we obtain that limit of solutions of equations (\ref{eq:representetieves system}) is a solution of integral equation  (\ref{eq:limit int}) with measure $\mu$ generated by $\sigma$, where $\sigma \in G$. The proof is complete.

\section{Remarks}
\begin{cor}
\label{cor: simple} Let $\delta$-sequence $\rho_n$ be of the
simplest type, what means that $\rho_n(t)=n\rho(nt)$, where $\rho \in C^\infty (\mathbb{R})$, $\rho \ge 0$, $\mathrm{supp} \ \rho \subseteq [0;1]$ and $\int_0^{1} \rho (s){\mathrm d}s =1$.

Then the sequence $F_n(F_n^{-1}(u)-\delta h_n )$
converges at any continuity point of the limit
function $\sigma$ for any $\delta \in (0;1)$ and the limit does not depend
on $\delta$ if and only if either $1/n = o(h_n)$ or $h_n =
o(1/n)$.

Moreover, suppose that $f$ is Lipschitz
function of a bounded growth and $\int_{t \in \mathbf{T}} |x_n^0(\tau_t)-x^0| d t \to 0$ as
$n \to \infty$ and $h_n \to 0$. Then $x_n$ -- solution of the equation (\ref{eq:representetieves system}) converges to $x(t)$ in $\mathbf{L_1(T)}$ if and only if
\begin{enumerate}
\item
$1/n = o(h_n)$, what implies that measure $\mu$, generated by $\sigma$, gives the mass one to the point $0$. Consecuently equation (\ref{eq: phi}) has the following solution
$$\varphi(z,x,u) = \left \{\begin{array}{l}
x, \ u =0, \\ x+z(x), \ u \in (0;1]. \\
\end{array} \right.$$
Thus $x(t)$ is a solution of the equation (\ref{eq:limit int}) that can be presented in the following form:
\begin{equation*}
\label{eq: Ito}
x(t)=x^0+\int_{(a,t]} f(s,x(s-))dL(s).
\end{equation*}

\item
$h_n = o(1/n)$, what means that measure $\mu$, generated by $\sigma$, is equal to the Lebesgue measure. Hence equation (\ref{eq: phi}) can be written as follows
\begin{equation*}
 \begin{split}
  &\frac{\partial \varphi(z,x,u)}{\partial u} = z (\varphi(z, x, u))\\
  &\varphi(z,x,0)=x.
 \end{split}
\end{equation*}
Thus $x(t)$ is a solution of the equation (\ref{eq:limit int}) with the function $\varphi$ defined above.
\end{enumerate}
\end{cor}



\end{document}